\documentclass[11pt,reqno]{amsart}

\usepackage[utf8]{inputenc}
\usepackage[english]{babel}
\usepackage{amsmath,amssymb,amsthm,mathtools}
\usepackage{xcolor}
\usepackage{geometry}
\usepackage{lmodern}

\DeclareMathOperator{\g}{\textsl{g}} 
\newcommand{\inner}[2]{\langle #1,#2\rangle} 
\newcommand{\tr}{\operatorname{tr}}
\renewcommand{\det}{\operatorname{det}}

\newcommand{\Aut}{\operatorname{Aut}}
\newcommand{\GL}{\operatorname{GL}}

\newcommand{\Sym}{\operatorname{Sym}}
\newcommand{\Herm}{\operatorname{Herm}}
\newcommand{\SO}{\operatorname{SO}}
\newcommand{\OO}{\operatorname{O}}

\newcommand{\Fix}{\operatorname{Fix}}
\newcommand{\Argmin}{\operatorname*{arg\,min}}
\newcommand{\bR}{\mathbb{R}}
\newcommand{\bC}{\mathbb{C}}

\newcommand{\gm}{\mathbin{\#}} 
\newcommand{\sgm}{\mathbin{\natural}} 

  \newtheorem{thmalph}{Theorem}

\newtheorem{theorem}{Theorem}[section]
\newtheorem{proposition}[theorem]{Proposition}
\newtheorem{lemma}[theorem]{Lemma}
\newtheorem{corollary}[theorem]{Corollary}
\newtheorem{definition}[theorem]{Definition}
\newtheorem{remark}[theorem]{Remark}
\newtheorem{example}[theorem]{Example}

\numberwithin{equation}{section}

 \usepackage[pagebackref,  colorlinks=true,  citecolor=cyan,
 citebordercolor={0 .255 .255},
  linkbordercolor={0 .255 .255},
  linkcolor=cyan
 ]{hyperref}

 \usepackage{cite}

 \def\equationautorefname~#1\null{(#1)\null}

\usepackage{filemod}

\title[Mostow's decompositions for symmetric cones]{ 
 Mostow-Type Decompositions and Geometric Means \\
 for Symmetric Cones
}
\author{Khalid Koufany}
\address{Universit\'e de Lorraine, CNRS, IECL, F-54000 Nancy, France}
\email{khalid.koufany@univ-lorraine.fr}
\author{Yongdo Lim}
\address{Department of Mathematics, Sungkyunkwan University, Suwan 440-746, Republic of Korea}
\email{ylim@knu.ac.kr}

 

\begin{document}

\begin{abstract}
Motivated by Mostow's decomposition theorem for positive definite matrices and by subsequent matrix factorizations involving geometric means, we establish a Mostow-type decomposition for arbitrary symmetric cones. As a consequence, we derive an analogous decomposition for the automorphism group of the cone. We also introduce a natural Hadamard metric on the symmetric cone by pulling back an $\ell^2$-product metric, together with its midpoint operation and associated Karcher mean.
\end{abstract}

\maketitle

\tableofcontents

\section{Introduction}

Mostow's 1955 decomposition theorem \cite[Theorem 5]{Mostow} may be viewed as a refinement of the polar decomposition for semisimple Lie groups. The following formulation is particularly close to the setting of this paper.
\begin{thmalph}[Mostow]\label{th:Mostow-group}
Let $G$ be a connected semisimple Lie group with Lie algebra $\mathfrak g$, let $\mathfrak{g}=\mathfrak{k}\oplus\mathfrak{p}$ be a Cartan decomposition, and let $K$ be the analytic subgroup of $G$ with Lie algebra $\mathfrak k$.
Suppose that $\mathfrak h\subset \mathfrak p$ is a subspace satisfying
\begin{equation*} 
[X,[X,Y]]\in \mathfrak h, \quad \forall X,Y\in \mathfrak h.
\end{equation*}
If $\mathfrak m=\mathfrak h^\perp$ denotes the orthogonal complement of $\mathfrak h$ with respect to the Killing form, then the multiplication map
$$
K\times \exp(\mathfrak m)\times \exp(\mathfrak h)\to G
$$
is a homeomorphism.
\end{thmalph}
Thus, Mostow's theorem splits the noncompact factor in the polar decomposition according to a Lie triple system and its orthogonal complement.   The theorem has since been revisited and extended in several geometric and analytic settings.   Loos proved a  closed result  in \cite{Loos}, where $\mathfrak{h}$ is $+1$-eigenspace of an involutive automorphism compatible with the Cartan decomposition.
Heinzner and Schwarz \cite{HeinznerSchwarz} obtained new proofs and extensions using moment map methods, while Porta and Recht \cite{PortaRecht} established an operator  analogue via conditional expectations. Further variants in nonpositively curved and infinite-dimensional settings were developed by Andruchow and Larotonda \cite{AndruchowLarotonda}, Conde and Larotonda \cite{CondeLarotonda}, Miglioli \cite{Miglioli}, and Tumpach \cite{TumpachLstar}.

For the cone $\Sym^{++}(n,\bR)$ of real symmetric positive definite matrices, Mostow's theorem takes the following concrete form; see Mostow \cite[Theorem 3]{Mostow}, Borel \cite[Proposition 1]{Borel}, or Lang \cite[Theorems 3.6 and 3.7, Chap.~XII]{Lang}.

\begin{thmalph}[Mostow]\label{MostowThm}
Let $E$ be a real subspace of $\Sym(n,\mathbb R)$. Then $\exp(E)$ is a totally geodesic and geodesically convex submanifold of $\Sym^{++}(n,\mathbb R)$ if and only if
\begin{equation}\label{eq:triple-system}
[X,[X,Y]]\in E, \quad  \forall  X,Y\in E.
\end{equation}
In this case every $X\in\Sym^{++}(n,\mathbb R)$ admits a unique factorization
\begin{equation}\label{eq:Mostow-Intro}
X=ABA,
\end{equation}
where $A\in\exp(E)$ and $B\in\exp(E^\perp)$, with $E^\perp$ taken with respect to the trace inner product. Moreover, the metric projection of $X$ onto $\exp(E)$ exists and is unique.
\end{thmalph}

This theorem shows that the geometry of totally geodesic submanifolds in the Cartan--Hadamard manifold $\Sym^{++}(n,\mathbb R)$ is closely related to canonical factorizations of positive definite matrices. The condition \eqref{eq:triple-system} is equivalent to $E$ being a Lie triple system and also to the closure condition
$$
efe\in \exp(E), \quad \forall e,f\in \exp(E).
$$
This viewpoint
has been revisited from matrix theory  perspectives; see,
for instance, \cite{Larotonda,TumpachLarotonda,Lim,Bhatia1}. In
particular, Lim \cite{Lim} showed that each Hermitian unitary matrix
induces a factorization of the cone of positive definite Hermitian
matrices into geodesically convex components. This decomposition is
naturally equipped with a Hadamard metric, for which one obtains
explicit midpoint formulas expressed in terms of both the metric
geometric mean and the spectral geometric mean. From a complementary
perspective, Bhatia \cite{Bhatia1} introduced the bipolar
decomposition, which provides an analogous factorization framework
for complex nonsingular matrices and highlights the underlying
geometric structure. Further developments include perturbative
analyses of Mostow's decomposition by Grover and Mishra
\cite{GroverMishra}, who established bounds quantifying its
stability under perturbations. In a different direction, Barbaresco
\cite{Barbaresco} applied Mostow's decomposition to the study of
fibrations of the unit Siegel disk, with applications to radar
space-time signal processing.

The natural setting for extending these ideas beyond matrices is the theory of symmetric cones. By the Koecher--Vinberg theorem, symmetric cones are precisely the cones of invertible squares in Euclidean Jordan algebras \cite{Koecher,Vinberg,FK}. Each symmetric cone carries a canonical invariant Riemannian metric. Lim introduced the metric geometric mean on symmetric cones \cite{LimGM}, and Lee and Lim later developed the spectral geometric mean \cite{LeeLim}. Symmetric cones therefore provide a natural Jordan algebra framework for a Mostow-type factorization.

In our setting, the role of Mostow's pair $(E,E^\perp)$ is played by the $\pm1$-eigenspace decomposition associated with an involutive Jordan automorphism. More precisely, let $V$ be a simple Euclidean Jordan algebra, let $\Omega$ be its symmetric cone, and let $\sigma$ be an involutive Jordan automorphism of $V$. We consider the fixed-point cone
$
\Omega_\sigma^+ = \{a\in\Omega:\sigma(a)=a\}
$
and the twisted fixed-point set
$
\Omega_\sigma^- = \{b\in\Omega:\sigma(b)=b^{-1}\}.
$
Our main theorem, see Theorem \ref{thm:factorization}, shows that the map
$$
f_\sigma:\Omega_\sigma^+\times \Omega_\sigma^- \to \Omega,
\qquad
f_\sigma(a,b)=P(a^{1/2})\,b,
$$
is a real-analytic diffeomorphism, with inverse
$$
x \longmapsto   \bigl(x\gm \sigma(x),x \sgm \sigma(x^{-1})\bigr),
$$
where $P(x)$ denotes the quadratic representation of $V$, while $a\gm b$ and $a\sgm b$ denote the metric and spectral geometric means, respectively. Thus, every $x\in \Omega$ admits a unique factorization
$$
x=P(a^{1/2})\,b,
\qquad
a\in\Omega_\sigma^+,\quad b\in\Omega_\sigma^-.
$$
This generalizes Theorem~\ref{MostowThm} to involutive Euclidean Jordan algebras. The decomposition also has a natural Riemannian interpretation: $\Omega_\sigma^+$ is the fixed-point set of the isometry $\sigma$, whereas $\Omega_\sigma^-$ is the fixed-point set of the isometry $\sigma\circ j$, where $j(x)=x^{-1}$. Consequently, both $\Omega_\sigma^+$ and $\Omega_\sigma^-$ are closed, totally geodesic submanifolds of the Hadamard manifold $\Omega$.

Let $G=G(\Omega)_0$ be the identity component of the cone automorphism group
$$
G(\Omega):=\{g\in\GL(V):g(\Omega)=\Omega\},
$$
and let $K=\Aut(V)_0$ be the identity component of the Jordan automorphism group. As a consequence of the Mostow decomposition of $\Omega$, we prove that the map
$$
 \Omega_\sigma^+\times \Omega_\sigma^-\times K\to G, \quad (a,b,k)\mapsto P(a^{1/2})P(b^{1/2})k
$$
is a real analytic diffeomorphism, with inverse
$$
g\mapsto
\Bigl(
x\gm\sigma(x),\,
x\sgm\sigma(x^{-1}),\,
P\bigl((x\sgm\sigma(x^{-1}))^{-1/2}\bigr)\,
P\bigl((x\gm\sigma(x))^{-1/2}\bigr)\,g
\Bigr),
$$
where $x=((gg^*)e)^{1/2}$.
This, in turn, extends Theorem~\ref{th:Mostow-group} to the reductive Lie group $G$.

Once the factorization is established, the geometry of $\Omega$ can be expressed in product form. Because $\Omega_\sigma^\pm$ are Hadamard spaces with their induced metrics, pulling back the $\ell^2$-product metric on $\Omega_\sigma^+\times \Omega_\sigma^-$ defines a new Hadamard metric $d_\sigma$ on $\Omega$. Its midpoint map and associated Karcher mean, defined as the unique minimizer of the sum of squared distances, are obtained by combining the corresponding constructions on the two factors componentwise. This part of the paper follows the barycentric viewpoint introduced by Karcher \cite{Karcher}. 

The paper is organized as follows. Section~\ref{sec:background} recalls the basic facts about Euclidean Jordan algebras and symmetric cones, and Section~\ref{sec:means} reviews the metric and spectral geometric means. In Section~\ref{sec:involutions}, we study involutive automorphisms and the geometry of $\Omega_\sigma^\pm$. Section~\ref{sec:factorization} contains the factorization theorem, the Mostow-type decomposition, and the metric projection formulas. Section~\ref{sec:MostowG} proves a Mostow decomposition for the identity component of the cone automorphism group. Sections~\ref{sec:l2metric} and~\ref{sec:karcher} introduce the induced $\ell^2$-metric and its associated Karcher mean.  
Finally, Section~\ref{sec:examples} treats the Hermitian, real symmetric, and Lorentz cones explicitly; in the Lorentz cone case, we obtain closed formulas for both geometric means.

\section{Symmetric cones and Euclidean Jordan algebras}\label{sec:background}

We briefly recall the Jordan algebra description of symmetric cones, referring to Faraut--Kor\'anyi \cite{FK} for proofs and further background.

Let $V$ be a finite-dimensional real Euclidean vector space with inner product $x,y \mapsto \inner{x}{y}$. A convex cone $\Omega \subset V$ is called \emph{pointed} if $\overline{\Omega}\cap (-\overline{\Omega})=\{0\}$. The group of linear automorphisms preserving $\Omega$ is
$$
G(\Omega)=\{g\in \GL(V): g(\Omega)=\Omega\}.
$$
It is a closed subgroup of the linear group $\GL(V)$ and hence a Lie group.
The cone $\Omega$ is called \emph{homogeneous} if $G(\Omega)$ acts transitively on $\Omega$, and \emph{self-dual} if the dual cone
$$
\Omega^* := \{x\in V : \inner{x}{y}>0 \text{ for all } y\in \overline{\Omega}\setminus \{0\}\}
$$
coincides with $\Omega$. A cone that is both homogeneous and self dual is called a \emph{symmetric cone}.

A \emph{Euclidean Jordan algebra} is a finite-dimensional Euclidean vector space $V$ endowed with a bilinear product $(x,y)\mapsto xy$ satisfying
$$
xy=yx,\qquad x(x^2y)=x^2(xy),\qquad \inner{xy}{z}=\inner{y}{xz}
$$
for all $x,y,z\in V$. By the Koecher--Vinberg theorem, if $V$ is a Euclidean Jordan algebra with unit element $e$, then
$$
\Omega = \operatorname{int}\{x^2 : x\in V\}
$$
is a symmetric cone; conversely, every symmetric cone arises this way \cite{FK,Koecher,Vinberg}.

For $x\in V$, let $L(x)$ and $P(x)$ denote the linear maps
$$
L(x)y = xy,
\qquad
P(x)=2L(x)^2-L(x^2).
$$
These operators are self adjoint with respect to the trace inner product. In particular, if $x\in\Omega$, then $P(x)$ is a positive definite linear operator on $V$. The operator $P(x)$ is called the \emph{quadratic representation}. Let $\Aut(V)$ denote the group of Jordan automorphisms of $V$. For $g\in G(\Omega)$, one has $g\in\Aut(V)$ if and only if $g(e)=e$. Moreover, every $g\in G(\Omega)$ can be written as $g=P(x)k$, with $x\in\Omega$ and $k\in\Aut(V)$.

From now on, we assume that $V$ is simple, meaning that it has no nontrivial ideals; consequently, $\Omega$ is irreducible.

An element $c\in V$ is an idempotent if $c^2=c$, and is primitive if it cannot be written as a sum of two nonzero idempotents. The rank of $V$ is the maximal number of pairwise orthogonal primitive idempotents. A \emph{Jordan frame} is a family $(c_1,\dots,c_r)$ of pairwise orthogonal primitive idempotents with $c_1+\cdots+c_r=e$.

Every $x\in V$ admits a spectral decomposition
$$
x = k\Bigl(\sum_{j=1}^r \lambda_j c_j\Bigr),
\qquad
k\in \Aut(V),\quad \lambda_j\in \bR.
$$
The numbers $\lambda_1,\dots,\lambda_r$ are the \emph{spectral values} of $x$. They define the Jordan \emph{determinant} and \emph{trace} by
$$
\det(x)=\prod_{j=1}^r \lambda_j,
\qquad
\tr(x)=\sum_{j=1}^r \lambda_j.
$$
The canonical associative inner product is
\begin{equation}\label{eq:associative-ip}
\inner{x}{y}=\tr(xy).
\end{equation}
If $x\in \Omega$, then all spectral values of $x$ are positive, and one defines
$$
x^\alpha = \sum_{j=1}^r \lambda_j^\alpha c_j,
\qquad
\log x = \sum_{j=1}^r (\log \lambda_j)c_j,
\qquad
\alpha\in \bR.
$$
The trace inner product is associative, and every Jordan automorphism is orthogonal with respect to it.

If $c\in V$ is an idempotent, then $L(c)$ is diagonalizable with eigenvalues $0$, $\frac12$, and $1$. Accordingly one has the Peirce decomposition
$$
V = V_1(c)\oplus V_{1/2}(c)\oplus V_0(c),
\qquad
V_\alpha(c)=\{x\in V : cx=\alpha x\}.
$$

An element $x\in V$ is invertible if there exists $y$ in the Jordan subalgebra generated by $\{x,e\}$ such that $xy=e$. The inverse is unique and is denoted by $x^{-1}$. Moreover, $x$ is invertible if and only if $\det x\ne0$. The set of invertible elements, denoted by $V^\times$, is open and dense in $V$.

\begin{example}{\rm
Two basic examples are the real symmetric matrices and the spin factors.

\begin{enumerate}
\item $V=\Sym(m,\bR)$ with Jordan product and the canonical inner product
$$
x\circ y = \frac12(xy+yx),\quad \inner{x}{y}=\tr(xy).
$$
In this case
$$
\Omega=\Sym^{++}(m,\bR),
$$
the cone of real symmetric positive definite matrices. The determinant and trace are the usual ones and
$$
P(x)y = xyx.
$$

\item $V=\bR\times \bR^n$ with product and the canonical inner product
$$
(\lambda,u)\circ (\mu,v) = \bigl(\lambda\mu+\inner{u}{v},\ \lambda v+\mu u\bigr),
\quad \inner{(\lambda,u)}{(\mu,v)}=\lambda\mu+\langle u,v\rangle,
$$
In this case,
$$
\Omega=\{(\lambda,u): \lambda>\|u\|\},
$$
the Lorentz cone. Here
$$
\tr(\lambda,u)=2\lambda,
\qquad
\det(\lambda,u)=\lambda^2-\|u\|^2,
\qquad
(\lambda,u)^{-1}=\frac{1}{\lambda^2-\|u\|^2}(\lambda,-u),
$$
and
$$
P(\lambda,u)(\mu,v)=\Bigl((\lambda^2+\|u\|^2)\mu+2\lambda\inner{u}{v},\ (\lambda^2-\|u\|^2)v+2(\lambda\mu+\inner{u}{v})u\Bigr).
$$
\end{enumerate}
}
\end{example}

The following standard facts will be used repeatedly.

\begin{proposition}[see {\cite[Propositions II.3.1, II.3.2, II.3.4]{FK}}]\label{prop:quadratic}
Let $V$ be a Euclidean Jordan algebra. Then:
\begin{enumerate}
\item $x\in V$ is invertible if and only if $P(x)$ is invertible. In this case
$$
P(x)x^{-1}=x,
\qquad
P(x)^{-1}=P(x^{-1}).
$$
\item The differential of the inversion map is
$$
D_v(x^{-1})=-P(x)^{-1}v.
$$
\item If $x,y\in V^\times$, then $P(x)y\in V^\times$ and
$$
\bigl(P(x)y\bigr)^{-1}=P(x^{-1})\,y^{-1}.
$$
\item One has the fundamental formula
$$
P(P(x)y)=P(x)\,P(y)\,P(x).
$$
\item For every $x\in V$,
$$
P(\exp x)=\exp\bigl(2L(x)\bigr).
$$
\end{enumerate}
\end{proposition}

The symmetric cone $\Omega$ carries a natural $G(\Omega)$-invariant Riemannian metric
$$
\g_x(u,v)=\inner{P(x^{-1})u}{v},
\qquad
x\in \Omega,\quad u,v\in V. 
$$
The corresponding distance is
$$
d(a,b)=\Bigl(\sum_{i=1}^r \log^2 \lambda_i\Bigr)^{1/2}, 
$$
where $\lambda_1,\dots,\lambda_r$ are the spectral values of $P(a^{-1/2})b$.  

With respect to this metric, the inversion map $j:x\mapsto x^{-1}$ is an involutive isometry of $\Omega$ whose unique fixed point is $e$. Given $y\in\Omega$, choose $g\in G(\Omega)$ such that $y=ge$. Then $s_g:=g\circ j\circ g^{-1}$ is a symmetry of $\Omega$ with unique fixed point $y$. Hence $\Omega$ is a Riemannian symmetric space of noncompact type. More precisely, it can be identified with the homogeneous space
$$
\Omega \simeq G / K ,
$$
where $G$ denotes the identity component of $G(\Omega)$ and
$
K = \{g\in G : ge=e\}
$
is the stabilizer of $e$. In particular, $\Omega$ is a Cartan--Hadamard manifold, that is, a simply connected, complete Riemannian manifold with nonpositive sectional curvature.

We conclude this section with the classification of simple Euclidean Jordan algebras and their corresponding irreducible symmetric cones:

 \begin{center}
 \begin{tabular}{|c|c|c|c|}
\hline ${V}$ & $\Omega$ & $\mathfrak{g}$ & $\mathfrak{k}$ \\
\hline $\operatorname{Sym}(n, \mathbb{R})$ & $\operatorname{Sym}^{++}(n, \mathbb{R})$ & $\mathfrak{s l}(n, \mathbb{R}) \oplus \mathbb{R}$ & $\mathfrak{o}(\mathrm{n})$ \\
$\operatorname{Herm}(n, \mathbb{C})$ & $\operatorname{Herm}^{++}(n, \mathbb{C})$ & $\mathfrak{s l}(n, \mathbb{C}) \oplus \mathbb{R}$ & $\mathfrak{s u}(\mathrm{n})$ \\
$\operatorname{Herm}(n, \mathbb{H})$ & $\operatorname{Herm}^{++}(n, \mathbb{H})$ & $\mathfrak{s l}(m, \mathbb{H}) \oplus \mathbb{R}$ & $\mathfrak{s u}(n, \mathbb{H})$ \\
$\mathbb{R} \times \mathbb{R}^{n-1}$ & $\Lambda_n$ & $\mathfrak{o}(1, n-1) \oplus \mathbb{R}$ & $\mathfrak{o}(\mathrm{n}-1)$ \\
$\operatorname{Herm}(3, \mathbb{O})$ & $\operatorname{Herm}^{++}(3, \mathbb{O})$ & $\mathfrak{e}_{(-26)} \oplus \mathbb{R}$ & $\mathfrak{f}_4$ \\
\hline
\end{tabular}
 \end{center}
Here $\mathfrak{g}$ and $\mathfrak{k}$ are the Lie algebras of $G$ and $K$, respectively.
\section{The metric and spectral geometric means}\label{sec:means}

For positive definite matrices $A,B$, the Kubo--Ando geometric mean is
$$
A\gm B = A^{1/2}\bigl(A^{-1/2}BA^{-1/2}\bigr)^{1/2}A^{1/2}. 
$$
It is the unique positive definite solution of the Riccati equation
$$
XA^{-1}X=B. 
$$
See \cite{KuboAndo,FiedlerPtak,BhatiaPDM}.

For a general symmetric cone $\Omega$ and $a,b\in \Omega$, define
$$
a\gm b := P(a^{1/2})\Bigl(P(a^{-1/2})b\Bigr)^{1/2}. 
$$
More generally, for $t\in \bR$,
$$
a\gm_t  b := P(a^{1/2})\Bigl(P(a^{-1/2})b\Bigr)^t.
$$

\begin{proposition}[see {\cite{LimGM}}]\label{prop:geodesic}
For $a,b\in \Omega$, the curve
$$
\gamma(t)=a\gm_t   b
$$
is the unique geodesic joining $a$ and $b$. In particular, $a\gm b$ is the midpoint of the geodesic segment from $a$ to $b$.
\end{proposition}

Thus $a\gm b$ is the \emph{metric geometric mean}. Its basic properties are as follows.

\begin{proposition}[see {\cite{LawsonLim,LimGM}}]\label{prop:metricmean}
Let $a,b\in \Omega$. Then:
\begin{enumerate}
\item $a\gm b$ is the unique solution  of the Riccati equation
$$
P(x)a^{-1}=b. 
$$
\item $a\gm b=b\gm a$.
\item $(a\gm b)^{-1}=a^{-1}\gm b^{-1}$.
\item $(\alpha a)\gm (\beta b)=\sqrt{\alpha\beta}\,(a\gm b)$ for $\alpha,\beta>0$.
\item $P(a\gm b)=P(a)\gm P(b)$, where the mean on the right-hand side is taken in the cone of positive definite selfadjoint operators on $V$.
\item $g(a\gm b)=g(a)\gm g(b)$ for every $g\in G(\Omega)$.
\item $\det(a\gm b)=\sqrt{\det(a)\det(b)}$.
\end{enumerate}
\end{proposition}

Following Lee and Lim \cite{LeeLim}, for $\alpha\in\bR$, define the weighted spectral geometric mean by
$$
a\sgm_\alpha b   := P\Bigl((a^{-1}\gm b)^\alpha\Bigr)a, \quad a, b\in \Omega,
$$
and in particular
$$
a\sgm b := a\sgm_{1/2} b.
$$
We call $a\sgm b$ the \emph{spectral geometric mean} of $a$ and $b$. For positive definite matrices, this definition reduces to the Fiedler--Pt\'ak spectral mean introduced in \cite{FiedlerPtak}:
$$
A\sgm B = \bigl(A^{-1}\gm B\bigr)^{1/2}A\bigl(A^{-1}\gm B\bigr)^{1/2}.
$$
Its spectrum consists of the positive square roots of the eigenvalues of $AB$; see \cite{FiedlerPtak,LeeLim}.

\begin{proposition}[see {\cite[Proposition 4.1]{LeeLim}}]\label{prop:spectralchar}
Let $a,b\in \Omega$ and $\alpha\in \bR$. Then $a\mathbin{\sgm_\alpha} b$ is the unique element $x\in \Omega$ such that
$$
a^{-1}\gm x = (a^{-1}\gm b)^\alpha. 
$$
\end{proposition}

The following proposition collects the basic properties of $\sgm_\alpha$.

\begin{proposition}[see {\cite[Section 4]{LeeLim}}]\label{prop:spectralprops}
Let $a,b\in \Omega$ and $\alpha\in \bR$. Then:
\begin{enumerate}
\item $a\mathbin{\sgm_0} b=a$ and $a\mathbin{\sgm_1} b=b$.
In particular, $\alpha\mapsto  a\sgm_\alpha b$ is a smooth
curve in $\Omega$ joining $a$ to $b$.

\item $(a\mathbin{\sgm_\alpha} b)^{-1}=a^{-1}\mathbin{\sgm_\alpha} b^{-1}$.
\item $a\mathbin{\sgm_\alpha} b=b\mathbin{\sgm_{1-\alpha}} a$. In particular, $a\sgm b=b\sgm a$.
\item $a^{-1}\gm (a\mathbin{\sgm_\alpha} b)=(a^{-1}\gm b)^\alpha$.
\item $x=(a^{-1}\gm b)^\alpha$ is uniquely determined by either of the
equivalent identities
$$
a\mathbin{\sgm_\alpha} b=P(x)\,a=P(x^{-1})\,b.
$$
\item $P(a\mathbin{\sgm_\alpha} b)=P(a)\mathbin{\sgm_\alpha} P(b)$, where the right-hand side is computed in the cone of positive definite selfadjoint operators.
\item $k(a\mathbin{\sgm_\alpha} b)=k(a)\mathbin{\sgm_\alpha} k(b)$ for every $k\in \Aut(V)$.
\item $\det(a\mathbin{\sgm_\alpha} b)=\det(a)^{1-\alpha}\det(b)^\alpha$.
\end{enumerate}
\end{proposition}

\section{Involutive automorphisms and the submanifolds $\Omega_\sigma^\pm$}\label{sec:involutions}
We now consider \emph{involutive automorphisms} of $V$, namely Jordan automorphisms $\sigma:V\to V$ satisfying $\sigma^2=\operatorname{id}$.

For such an automorphism, consider the $\pm1$-eigenspace decomposition
$$
V^+ := \{x\in V : \sigma(x)=x\},
\qquad
V^- := \{x\in V : \sigma(x)=-x\}.
$$
Since $\sigma$ preserves the trace inner product, one has the orthogonal decomposition
$$
V=V^+\oplus^\perp V^-.
$$
Moreover,
$$
V^+V^+ \subset V^+,
\qquad
V^-V^- \subset V^+,
\qquad
V^+V^- \subset V^-.
$$
Thus $V^+$ is a Jordan subalgebra, while $V^-$ is a Jordan triple system.

A particularly important class of involutions is provided by quadratic representations.

\begin{proposition}[see {\cite[Proposition II.4.4]{FK}}]\label{prop:special}
For $x\in V$, the map $P(x)$ is an involutive automorphism of $V$ if and only if $x$ is an involutive element, that is, $x^2=e$.
\end{proposition}

Such an involution will be called \emph{special}.

\begin{remark}{\rm
\begin{enumerate}
\item If $w^2=e$ and $w=\sum_{j=1}^r \lambda_j c_j$ is its spectral decomposition, then $\lambda_j=\pm 1$ for every $j$.

\item There is a bijection between idempotents $c$ and involutive elements $w$ given by
$$
w=2c-e,
\qquad
c=\frac12(w+e).
$$
\item Let $c$ be an idempotent of $V$ and  $\sigma=P(w)$ the corresponding special involutive automorphism (with $w=2c-e$). Then
$$
V^+=V_0(c)\oplus V_1(c),
\qquad
V^-=V_{1/2}(c).
$$
\end{enumerate}
}
\end{remark}

Special involutive automorphisms provide the natural framework for extending the results of \cite{Lim}. We shall nevertheless work with arbitrary involutive automorphisms.

Fix an involutive automorphism $\sigma$ and define
$$
\Omega_\sigma^+ := \{a\in \Omega : \sigma(a)=a\},
\qquad
\Omega_\sigma^- := \{b\in \Omega : \sigma(b)=b^{-1}\}.
$$

\begin{proposition}\label{prop:exp}
One has
$$
\Omega_\sigma^+=\exp(V^+),
\qquad
\Omega_\sigma^-=\exp(V^-).
$$
In particular, $\Omega_\sigma^\pm$ are simply connected, smooth, embedded submanifolds of $\Omega$.
\end{proposition}

\begin{proof}
Since $\sigma$ is a Jordan automorphism, it commutes with the Jordan exponential:
$$
\sigma(\exp x)=\sum_{n\ge 0}\frac{\sigma(x)^n}{n!}=\exp(\sigma(x)).
$$
If $x\in V^+$, then $\sigma(x)=x$ and hence $\sigma(\exp x)=\exp x$, so $\exp x\in \Omega_\sigma^+$. Conversely, if $a\in \Omega_\sigma^+$ and $a=\exp x$, then
$$
\exp(\sigma(x))=\sigma(\exp x)=\sigma(a)=a=\exp x.
$$
Since $\exp:V\to \Omega$ is injective, $\sigma(x)=x$, so $x\in V^+$.

Similarly, if $x\in V^-$, then $\sigma(x)=-x$, hence
$$
\sigma(\exp x)=\exp(-x)=(\exp x)^{-1},
$$
so $\exp x\in \Omega_\sigma^-$. Conversely, if $b\in \Omega_\sigma^-$ and $b=\exp x$, then
$$
\exp(\sigma(x))=\sigma(\exp x)=\sigma(b)=b^{-1}=\exp(-x),
$$
hence $\sigma(x)=-x$, so $x\in V^-$.

Since $\exp:V\to \Omega$ is a diffeomorphism, the final assertion follows.
\end{proof}

\begin{lemma}\label{lem:rho}
The map
$$
\rho:=\sigma\circ j : \Omega\to \Omega,
\qquad
\rho(x)=\sigma(x^{-1}), 
$$
is an isometry of $(\Omega,d)$.
\end{lemma}

\begin{proof}
The inversion map $j:x\mapsto x^{-1}$ and every Jordan automorphism are isometries of $(\Omega,d)$, so their composition is again an isometry.
\end{proof}

\begin{proposition}\label{prop:totgeod}
The submanifolds $\Omega_\sigma^+$ and $\Omega_\sigma^-$ are closed, totally geodesic, and geodesically convex in $\Omega$. In particular, with the induced metric they are Hadamard manifolds.
\end{proposition}

\begin{proof}
First, $\Omega_\sigma^+=\operatorname{Fix}(\sigma)$ is the fixed-point set of the isometry $\sigma$. A standard result in Riemannian geometry states that the fixed-point set of an isometry of a Cartan--Hadamard manifold is a closed, totally geodesic submanifold; see, for example, \cite[Theorem 1.10.15]{Kling} or \cite{Ballmann-NPC,Jost-NPC}. For completeness, let $x,y\in\Omega_\sigma^+$ and let $\gamma$ be the geodesic joining them. Since $\sigma$ is an isometry, $\sigma\circ\gamma$ is another geodesic joining $x$ and $y$. Uniqueness of geodesics in a Hadamard manifold gives $\sigma(\gamma(t))=\gamma(t)$, so $\gamma(t)\in\Omega_\sigma^+$ for every $t$. Thus $\Omega_\sigma^+$ is also geodesically convex.

The same argument applies to $\Omega_\sigma^-=\Fix(\rho)$ by Lemma~\ref{lem:rho}. In a Hadamard manifold, a closed totally geodesic submanifold is geodesically convex.

\end{proof}

\begin{remark}{\rm
\begin{enumerate}
\item
$\Omega_\sigma^+$ and $\Omega_\sigma^-$ are symmetric spaces of noncompact type. Indeed,
 $$
 \Omega_\sigma^+\simeq G^\sigma/K^\sigma,
 \qquad
 \Omega_\sigma^-\simeq G^\rho/K^\rho,
 $$
 where $G^\sigma=\{g\in G:\sigma\circ g\circ\sigma=g\}$, $K^\sigma$ is the stabilizer of $e$ in $G^\sigma$, and $G$ is the identity component of $G(\Omega)$. The groups $G^\rho$ and $K^\rho$ are defined analogously.

 \item Moreover, $\Omega_\sigma^+$ is the symmetric cone of the Euclidean Jordan algebra $V^+$, whereas $\Omega_\sigma^-$ is a Makarevi\v c space of noncompact type; see \cite{Bertram}.
 \end{enumerate}
 }
 \end{remark}

\begin{lemma}\label{lem:detone}
If $b\in \Omega_\sigma^-$, then $\det(b)=1$.
\end{lemma}

\begin{proof}
Since Jordan automorphisms preserve the determinant,
$$
\det(b)=\det(\sigma(b))=\det(b^{-1})=\det(b)^{-1}.
$$
As $\det(b)>0$, it follows that $\det(b)=1$.
\end{proof}

\section{Mostow's decomposition of $\Omega$}\label{sec:factorization}
From now on, $V$ is a simple Euclidean Jordan algebra, $\Omega$ is its symmetric cone, and $\sigma$ is an involutive automorphism of $V$.

Define
$$
f_\sigma : \Omega_\sigma^+\times \Omega_\sigma^- \to \Omega,
\qquad
f_\sigma(a,b)=P(a^{1/2})\,b.
$$

\begin{theorem}[Mostow-type factorization]\label{thm:factorization}
The map $f_\sigma$ is a real-analytic diffeomorphism. Its inverse is
\begin{equation}\label{eq:inv-form}
f_\sigma^{-1}(x)=\bigl(x\gm \sigma(x),\ x\sgm \sigma(x^{-1})\bigr).
\end{equation}
\end{theorem}

\begin{proof}
Let $x\in \Omega$ and set
$$
a:=x\gm \sigma(x).
$$
By Proposition~\ref{prop:metricmean}(2) and (6),
$$
\sigma(a)=\sigma(x)\gm \sigma^2(x)=\sigma(x)\gm x=x\gm \sigma(x)=a,
$$
so $a\in \Omega_\sigma^+$.

Now define
$$
b:=P(a^{-1/2})\,x.
$$
We show that $b\in \Omega_\sigma^-$. Since $a=x\gm \sigma(x)$, Proposition~\ref{prop:metricmean}(1) gives
$$
P(a)x^{-1}=\sigma(x).
$$
Using the fact that $\sigma$ commutes with $P$, we obtain
$$
\sigma(b)=\sigma\bigl(P(a^{-1/2})x\bigr)=P(a^{-1/2})\sigma(x)=P(a^{-1/2})P(a)x^{-1}=P(a^{1/2})x^{-1}.
$$
On the other hand, Proposition~\ref{prop:quadratic}(3) yields
$$
b^{-1}=\bigl(P(a^{-1/2})x\bigr)^{-1}=P(a^{1/2})x^{-1}.
$$
Hence $\sigma(b)=b^{-1}$ and therefore $b\in \Omega_\sigma^-$.

Also,
$$
f_\sigma(a,b)=P(a^{1/2})P(a^{-1/2})x=x.
$$
To identify the second component, note that by Proposition~\ref{prop:metricmean}(3),
$$
a^{-1}=(x\gm \sigma(x))^{-1}=x^{-1}\gm \sigma(x^{-1}).
$$
Therefore, by the definition of the spectral mean,
$$
x\sgm \sigma(x^{-1})
=
P\Bigl((x^{-1}\gm \sigma(x^{-1}))^{1/2}\Bigr)x
=
P(a^{-1/2})x
=
b.
$$
Thus $f_\sigma$ is surjective and the map in \eqref{eq:inv-form} is its right inverse.

To prove injectivity, suppose
$$
x=f_\sigma(a_1,b_1)=f_\sigma(a_2,b_2),
\qquad
(a_i,b_i)\in \Omega_\sigma^+\times \Omega_\sigma^-.
$$
Fix any representation $x=P(a^{1/2})b$ with $a\in \Omega_\sigma^+$ and $b\in \Omega_\sigma^-$. Since $\sigma(a)=a$ and $\sigma(b)=b^{-1}$,
$$
\sigma(x)=P(a^{1/2})b^{-1}.
$$
Also
$$
x^{-1}=\bigl(P(a^{1/2})b\bigr)^{-1}=P(a^{-1/2})b^{-1},
$$
so
$$
P(a)x^{-1}=P(a^{1/2})b^{-1}=\sigma(x).
$$
By Proposition~\ref{prop:metricmean}(1), the unique solution of this Riccati equation is $a=x\gm \sigma(x)$. Hence $a$ is uniquely determined by $x$, and therefore $a_1=a_2$. Once $a$ is known,
$$
b=P(a^{-1/2})x
$$
is uniquely determined as well. So $b_1=b_2$, proving injectivity.

Finally, $f_\sigma$ is clearly real analytic. The right-hand side of \eqref{eq:inv-form} is a composition of real analytic maps on $\Omega$ (the geometric mean, inversion, the square root, and $P$) so $f_\sigma^{-1}$ is also real analytic. Therefore $f_\sigma$ is a real analytic diffeomorphism.
\end{proof}

For $x\in\Omega$, define
$$
\pi_\sigma^+(x):=x\gm\sigma(x),
\qquad
\pi_\sigma^-(x):=x\gm\sigma(x^{-1}).
$$
By the preceding discussion, $\pi_\sigma^\pm(x)\in\Omega_\sigma^\pm$.

\begin{proposition}\label{prop:projection}
For every $x\in\Omega$, one has
$$
\pi_\sigma^\pm(x)=\Argmin_{y\in \Omega_\sigma^\pm} d(x,y).
$$
In other words, $\pi_\sigma^\pm(x)$ is the metric projection of $x$ onto $\Omega_\sigma^\pm$.
\end{proposition}

\begin{proof}
The proof is the usual midpoint argument for fixed-point sets of isometries in a Hadamard manifold.

For the $+$-case, let $x\in \Omega$ and $y\in \Omega_\sigma^+$. Since $\pi_\sigma^+(x)=x\gm \sigma(x)$ is the midpoint of the geodesic from $x$ to $\sigma(x)$,
$$
d\bigl(x,\pi_\sigma^+(x)\bigr)=\frac12\,d\bigl(x,\sigma(x)\bigr).
$$
By the triangle inequality,
$
d\bigl(x,\sigma(x)\bigr)\le d(x,y)+d\bigl(y,\sigma(x)\bigr).
$
Since $\sigma(y)=y$ and $\sigma$ is an isometry, we have
$
d\bigl(y,\sigma(x)\bigr)=d\bigl(\sigma(y),\sigma(x)\bigr)=d(y,x).
$
Hence
$$
d\bigl(x,\pi_\sigma^+(x)\bigr)\le d(x,y),
\qquad \forall y\in\Omega_\sigma^+.
$$
Thus $\pi_\sigma^+(x)$ minimizes the map $y\mapsto d(x,y)$ on $\Omega_\sigma^+$.

The proof for the $-$-case is identical, using the isometry $\rho=\sigma\circ j$ from Lemma~\ref{lem:rho}. Indeed, if $y\in \Omega_\sigma^-$, then $\rho(y)=y$, so
$$
d\bigl(y,\sigma(x^{-1})\bigr)=d\bigl(\rho(y),\rho(x)\bigr)=d(y,x).
$$
Therefore the midpoint of the geodesic from $x$ to $\sigma(x^{-1})$ minimizes the distance from $x$ to $\Omega_\sigma^-$. Finally, by Proposition~\ref{prop:totgeod}, the metric projection is unique, and thus
$$
\pi_\sigma^\pm(x)
=
\Argmin_{y\in\Omega_\sigma^\pm} d(x,y).
$$

\end{proof}

 \begin{remark}{\rm
 It is convenient to distinguish between the factorization components and the metric projections. We therefore set
\begin{equation}\label{eq:fact-compo}
q_\sigma^+(x):=x\gm \sigma(x),
\qquad
q_\sigma^-(x):=x\sgm \sigma(x^{-1}),
\end{equation}
so that $f_\sigma^{-1}(x)=\bigl(q_\sigma^+(x),q_\sigma^-(x)\bigr)$.

Notice that $q_\sigma^+(x)=\pi_\sigma^+(x)$, whereas $q_\sigma^-(x)$ and $\pi_\sigma^-(x)$ generally differ.
}
\end{remark}

\begin{corollary}[Mostow decomposition]\label{cor:mostow}
Every $x\in \Omega$ can be written uniquely as
$$
x=P(a^{1/2})\,b
$$
with $a\in \Omega_\sigma^+$ and $b\in \Omega_\sigma^-$.
\end{corollary}

Since $\Omega_\sigma^\pm=\exp(V^\pm)$ by Proposition~\ref{prop:exp}, one also obtains:

\begin{corollary}[Mostow decomposition]\label{cor:mostow-exp}
Every $x\in \Omega$ can be written uniquely in the form
$$
x=P\bigl((\exp d)^{1/2}\bigr)\exp(v)
$$
with $d\in V^+$ and $v\in V^-$.
\end{corollary}

\begin{proof}
The map $d\mapsto\exp d$ is a diffeomorphism from $V^+$ onto $\Omega_\sigma^+$, the map $v\mapsto\exp v$ is a diffeomorphism from $V^-$ onto $\Omega_\sigma^-$, and $f_\sigma$ is a diffeomorphism from $\Omega_\sigma^+\times\Omega_\sigma^-$ onto $\Omega$. Consequently, the map
$$
\Phi : V^+ \times V^- \to \Omega, \qquad
\Phi(d,v) = P\big( (\exp d)^{1/2}\big) \exp(v)
$$
is a real-analytic diffeomorphism.
\end{proof}

This is the symmetric cone analogue of the classical Mostow decomposition \eqref{eq:Mostow-Intro} stated in Theorem~\ref{MostowThm}.

\section{Mostow's decomposition of $G(\Omega)_0$}\label{sec:MostowG}

Let $G:=G(\Omega)_0$ be the identity component of $G(\Omega)$, and let $K:=\Aut(V)_0$ be the identity component of $\Aut(V)$. Then
$$
K=G\cap O(V),
$$
where $O(V)$ is the orthogonal group of $V$. Recall that $G$, and more generally $G(\Omega)$, is a reductive Lie group with Cartan involution $g\mapsto g^{-1*}$, where $g^*$ denotes the adjoint of $g$ with respect to the associative inner product \eqref{eq:associative-ip}.

By the polar decomposition of $G$ \cite[Theorem III.5.1]{FK}, every $g\in G$ can be written uniquely as $g=P(x)k$, with $x\in\Omega$ and $k\in K$. The next theorem refines this decomposition by applying the Mostow factorization to the cone variable $x$.

 \begin{theorem}[Mostow's decomposition of $G$]
The map
$$
\Psi_\sigma:\Omega_\sigma^+\times \Omega_\sigma^-\times K\to G,
\qquad
\Psi_\sigma(a,b,k)=P(a^{1/2})\,P(b^{1/2})\,k,
$$
is a real analytic diffeomorphism.

More explicitly, for $g\in G$, let
$$
x:=\bigl((gg^*)e\bigr)^{1/2}\in\Omega,
$$
and define
$$
a:=x\gm \sigma(x)\in\Omega_\sigma^+,
\qquad
b:=x\sgm\sigma(x^{-1})\in\Omega_\sigma^-.
$$
Then
$$
\Psi_\sigma^{-1}(g)
=
\Bigl(
a,\,
b,\,
P(b^{-1/2})\,P(a^{-1/2})\,g
\Bigr).
$$
Equivalently, every $g\in G$ admits a unique factorization
$$
g=P(a^{1/2})\,P(b^{1/2})\,k,
\qquad
a\in \Omega_\sigma^+,
\quad
b\in \Omega_\sigma^-,
\quad
k\in K.
$$
\end{theorem}

\begin{proof}
Let $g\in G$. Since $g^*\in G$,  
$$
q:=gg^*
$$
belongs to $G$ and is positive self adjoint. By the polar decomposition of $G$, there
exists a unique $x\in\Omega$ such that
$
q=P(x).
$
Hence,
$
x=\bigl((gg^*)e\bigr)^{1/2}.
$
By Theorem~\ref{thm:factorization}, there exist unique
$
a\in\Omega_\sigma^+$ and
$
b\in\Omega_\sigma^-,
$
such that
$
x=P(a^{1/2})\,b.
$
Set
$$
h:=P(a^{1/2}),
\qquad
f:=P(b^{1/2}).
$$
Then, by the fundamental formula for quadratic representations,
$$
q=P(x)=P(P(a^{1/2})b)=P(a^{1/2})\,P(b)\,P(a^{1/2})=h\,f^2\,h,
$$
where $h=P(a^{1/2})$ and $f=P(b^{1/2})$.
Since $h$ and $f$ are self adjoint,
$
h\,f^2\,h=(hf)(hf)^*.
$
Now define
$$
k:=(hf)^{-1}g=P(b^{-1/2})\,P(a^{-1/2})\,g.
$$
Then
$$
kk^*=(hf)^{-1}\,gg^*\,(hf)^{-1*}
=(hf)^{-1}\,q\,(hf)^{-1*}
=I.
$$
Thus $k$ is orthogonal. Since $k\in G$, it follows that
$
k\in G\cap O(V)=K.
$
Therefore
$$
g=hfk=P(a^{1/2})\,P(b^{1/2})\,k,
$$
which proves surjectivity.

To prove injectivity, suppose
$$
P(a_1^{1/2})P(b_1^{1/2})k_1
=
P(a_2^{1/2})P(b_2^{1/2})k_2.
$$
Applying both sides to the unit element $e$ yields
$$
P(a_1^{1/2})b_1=P(a_2^{1/2})b_2.
$$
By uniqueness of the Mostow decomposition of $\Omega$, we obtain $a_1=a_2$ and $b_1=b_2$, and then also $k_1=k_2$. Hence $\Psi_\sigma$ is bijective.

Finally, $\Psi_\sigma$ is real analytic, and its inverse is real analytic because
it is obtained from the real analytic operations
$$
g\mapsto gg^*,\qquad
q\mapsto (qe)^{1/2},\qquad
x\mapsto x\gm \sigma(x),\qquad
x\mapsto x\sgm\sigma(x^{-1}),
$$
together with multiplication and inversion in $G$.
\end{proof}

\begin{remark}
    {\rm
One can also prove that the map
$$
\widetilde{\Psi}_\sigma:K\times \Omega_\sigma^- \times \Omega_\sigma^+ \to G,
\qquad
\widetilde{\Psi}_\sigma(k,b,a)=k\,P(b^{1/2})\,P(a^{1/2}),
$$
is a real analytic diffeomorphism, with inverse
$$
\widetilde{\Psi}_\sigma^{-1}(g)=\bigl(g\,P(a^{-1/2})\,P(b^{-1/2}),\, b,\, a\bigr),
$$
where
$$
x:=\bigl((g^*g)e\bigr)^{1/2},
\qquad
a:=x\gm \sigma(x)\in \Omega_\sigma^+,
\qquad
b:=x\sgm \sigma(x^{-1})\in \Omega_\sigma^-.
$$
Equivalently, every $g\in G$ admits a unique factorization
$$
g=k\,P(b^{1/2})\,P(a^{1/2}),
\qquad
k\in K,\quad b\in \Omega_\sigma^-,\quad a\in \Omega_\sigma^+.
$$
    }
\end{remark}

\section{The pulled-back $\ell^2$-metric and its midpoint map}\label{sec:l2metric}

Let $d^\pm=d_{\mid\Omega_\sigma^\pm}$ denote the intrinsic distances on the Hadamard manifolds $\Omega_\sigma^\pm$. For
$$
x=f_\sigma(a_1,b_1),
\qquad
y=f_\sigma(a_2,b_2),
$$
define
$$
d_\sigma(x,y):=\Bigl(d^+(a_1,a_2)^2+d^-(b_1,b_2)^2\Bigr)^{1/2}.
$$

\begin{proposition}\label{thm:dsigma}
The metric space $(\Omega,d_\sigma)$ is Hadamard. The map
$$
f_\sigma : \bigl(\Omega_\sigma^+\times \Omega_\sigma^-,\ ((a_1,b_1),(a_2,b_2))\mapsto (d^+(a_1,a_2)^2+d^-(b_1,b_2)^2)^{1/2}\bigr) \to (\Omega,d_\sigma)
$$
is an isometry.

Moreover, if $x=f_\sigma(a_1,b_1)$ and $y=f_\sigma(a_2,b_2)$, then the unique $d_\sigma$-geodesic joining $x$ and $y$ is
$$
m_\sigma(t;x,y):=f_\sigma(a_1\gm_t  a_2,\ b_1\gm_t  b_2),
\qquad 0\le t\le 1.
$$
In particular, the midpoint is
$$
m_\sigma(x,y):=m_\sigma\Bigl(\frac12;x,y\Bigr)=f_\sigma(a_1\gm a_2,\ b_1\gm b_2).
$$
That is,
$$
m_\sigma(x,y)=P(a_*^{1/2})\,b_*,
$$
where
$$a_*:=\left(x\gm\sigma(x)\right)\gm\left(y\gm\sigma(y)\right),\;\;
b_*:=\left(x\sgm\sigma(x^{-1})\right)\gm\left(y\sgm\sigma(y^{-1})\right).$$
  
\end{proposition}

\begin{proof}
The $\ell^2$-product of the Hadamard spaces $(\Omega_\sigma^\pm,d^\pm)$ is again Hadamard for the metric
$$
\delta((a_1,b_1),(a_2,b_2))
:=\bigl(d^+(a_1,a_2)^2+d^-(b_1,b_2)^2\bigr)^{1/2}.
$$
This follows by adding the semiparallelogram inequalities in the two factors. Since $f_\sigma$ is a bijection, the definition of $d_\sigma$ is well posed, and
$$
d_\sigma(f_\sigma(a_1,b_1),f_\sigma(a_2,b_2))
=\delta((a_1,b_1),(a_2,b_2)).
$$
Thus $f_\sigma$ is an isometry from the product Hadamard space $(\Omega_\sigma^+\times\Omega_\sigma^-,\delta)$ onto $(\Omega,d_\sigma)$, and $(\Omega,d_\sigma)$ is Hadamard.

In the product space, the unique geodesic from $(a_1,b_1)$ to $(a_2,b_2)$ is
$$
t\longmapsto (a_1\gm_t a_2 , b_1\gm_t b_2),
$$
because $\Omega_\sigma^\pm$ are totally geodesic in $(\Omega,d)$ and their intrinsic geodesics
are the ambient weighted geometric means. Transporting this geodesic by the isometry $f_\sigma$
gives the stated formula for $m_\sigma(t;x,y)$, and hence for the midpoint $m_\sigma(x,y)$.

 \end{proof}

\begin{proposition}\label{prop:properties-msigma}
For all $x,y,x',y'\in\Omega$ and all $\alpha,\beta>0$, the mean $m_\sigma$ has the following properties:

 \begin{enumerate}
\item $m_\sigma(x,x)=x$;

\item $m_\sigma(x,y)=m_\sigma(y,x)$;

\item $m_\sigma(\alpha x,\beta y)=\sqrt{\alpha\beta}\,m_\sigma(x,y)$;

\item
$
d_\sigma\bigl(m_\sigma(x,y),m_\sigma(x',y')\bigr)
\le \frac12\bigl(d_\sigma(x,x')+d_\sigma(y,y')\bigr).
$
More generally, for every $t\in[0,1]$,
$$
d_\sigma\bigl(m_\sigma(t;x,y),m_\sigma(t;x',y')\bigr)
\le
(1-t)\,d_\sigma(x,x')+t\,d_\sigma(y,y').
$$

\item
$
m_\sigma(x^{-1},y^{-1})^{-1}=m_\sigma(x,y);
$

\item
$
\det(m_\sigma(x,y))=\sqrt{\det(x)\det(y)}.
$
\end{enumerate}

 \end{proposition}

\begin{proof}
All assertions are inherited from the corresponding properties of the weighted geometric
mean on the Hadamard factors $(\Omega_\sigma^\pm,d^\pm)$.

Indeed, write
$
x=f_\sigma(a_1,b_1)
$
and
$
y=f_\sigma(a_2,b_2).
$
Then
$$
m_\sigma(t;x,y)=f_\sigma(a_1\gm_t a_2,\ b_1\gm_t b_2).
$$
Items (1) and (2) are immediate.

For (3), observe that
$$
q_\sigma^+(\alpha x)=\alpha q_\sigma^+(x),
\qquad
q_\sigma^-(\alpha x)=q_\sigma^-(x),
$$
and use the homogeneity of $\gm$ together with
$$
f_\sigma(\lambda a,b)=\lambda f_\sigma(a,b).
$$

For (4), geodesic interpolation in any Hadamard space satisfies the joint convexity estimate
$$
d^+\bigl(a_1\gm_t a_2,\ a_1'\gm_t a_2'\bigr)
\le
(1-t)\,d^+(a_1,a_1')+t\,d^+(a_2,a_2'),
$$
and similarly on $\Omega_\sigma^-$. Combining these two estimates and using the definition of $d_\sigma$ gives the stated inequality.

For (5), one has
$$
x^{-1}=f_\sigma(a_1^{-1},b_1^{-1}),
\qquad
y^{-1}=f_\sigma(a_2^{-1},b_2^{-1}),
$$
and Proposition~\ref{prop:metricmean}(3) yields
$$
a_1^{-1}\gm a_2^{-1}=(a_1\gm a_2)^{-1},
\qquad
b_1^{-1}\gm b_2^{-1}=(b_1\gm b_2)^{-1}.
$$
Hence
$$
m_\sigma(x^{-1},y^{-1})=m_\sigma(x,y)^{-1}.
$$

Finally, by Lemma~\ref{lem:detone},
$$
\det(x)=\det(a_1),
\qquad
\det(y)=\det(a_2).
$$
Therefore
$$
\det\bigl(m_\sigma(x,y)\bigr)
=
\det(a_1\gm a_2)\det(b_1\gm b_2)
=
\sqrt{\det(a_1)\det(a_2)}
=
\sqrt{\det(x)\det(y)}.
$$

\end{proof}

\begin{remark}{\rm
If $\sigma=\operatorname{id}$, then $\Omega_\sigma^-=\{e\}$, $d_\sigma=d$, and $m_\sigma(x,y)=x\gm y$. More generally, if $x,y\in\Omega_\sigma^+$ or $x,y\in\Omega_\sigma^-$, then
$m_\sigma(x,y)={x}\gm{y}$.
}
\end{remark}

\section{The Karcher mean associated with $\sigma$}\label{sec:karcher}

For $a_1,\dots,a_n\in \Omega_\sigma^+$ and $b_1,\dots,b_n\in \Omega_\sigma^-$, let
$$
\Lambda_n^+(a_1,\dots,a_n)
:=
\Argmin_{a\in \Omega_\sigma^+}\frac1n\sum_{i=1}^n d^+(a,a_i)^2
$$
and
$$
\Lambda_n^-(b_1,\dots,b_n)
:=
\Argmin_{b\in \Omega_\sigma^-}\frac1n\sum_{i=1}^n d^-(b,b_i)^2.
$$
These minimizers exist and are unique because $\Omega_\sigma^\pm$ are Hadamard spaces \cite{Karcher,BH}.

\begin{definition}
For $x_i=f_\sigma(a_i,b_i)\in\Omega$, $i=1,\dots,n$, define
$$
\Lambda_n^\sigma(x_1,\dots,x_n)
:=
f_\sigma\bigl(\Lambda_n^+(a_1,\dots,a_n),\ \Lambda_n^-(b_1,\dots,b_n)\bigr).
$$
In other words,
$$
\Lambda_n^\sigma(x_1,\dots,x_n)
=
P(A_*^{1/2})\,B_*,
$$
where
$$
A_*:=\Lambda_n^+\left({x_1}\gm{\sigma(x_1)},\dots, {x_n}\gm{\sigma(x_n)}\right),
\quad
B_*:=\Lambda_n^-\left({x_1}\sgm{\sigma(x_1^{-1})},\dots, {x_n}\sgm{\sigma(x_n^{-1})}\right).
$$
\end{definition}

\begin{theorem}\label{thm:karcher}
The mean $\Lambda_n^\sigma$ is the Karcher mean in the Hadamard space $(\Omega,d_\sigma)$; namely,
$$
\Lambda_n^\sigma(x_1,\dots,x_n)
=
\Argmin_{x\in \Omega}\frac1n\sum_{i=1}^n d_\sigma(x,x_i)^2.
$$
\end{theorem}

\begin{proof}
Let $x=f_\sigma(a,b)$. Since $f_\sigma$ is an isometry for $d_\sigma$,
$$
d_\sigma(x,x_i)^2 = d^+(a,a_i)^2+d^-(b,b_i)^2.
$$
Hence
$$
\frac1n\sum_{i=1}^n d_\sigma(x,x_i)^2
=
\frac1n\sum_{i=1}^n d^+(a,a_i)^2
+
\frac1n\sum_{i=1}^n d^-(b,b_i)^2.
$$
The two terms are minimized independently at $\Lambda_n^+(a_1,\dots,a_n)$ and $\Lambda_n^-(b_1,\dots,b_n)$, respectively.
\end{proof}

\begin{proposition}\label{prop:karcherprops}
For all $x_1,\dots,x_n\in \Omega$ and all $\alpha_1,\dots,\alpha_n>0$:
\begin{enumerate}
\item $\Lambda_n^\sigma(x,\dots,x)=x$;
\item $\Lambda_2^\sigma(x,y)=m_\sigma(x,y)$;
\item $\Lambda_n^\sigma(x_{\tau(1)},\dots,x_{\tau(n)})=\Lambda_n^\sigma(x_1,\dots,x_n)$ for every permutation $\tau$;
\item
$
d_\sigma\bigl(\Lambda_n^\sigma(x_1,\dots,x_n),\Lambda_n^\sigma(y_1,\dots,y_n)\bigr)
\le
\frac1n\sum_{i=1}^n d_\sigma(x_i,y_i);
$

\item
$
\Lambda_n^\sigma(\alpha_1x_1,\dots,\alpha_n x_n)
=
(\alpha_1\cdots \alpha_n)^{1/n}\Lambda_n^\sigma(x_1,\dots,x_n);
$
\item
$
\Lambda_n^\sigma(x_1^{-1},\dots,x_n^{-1})^{-1}
=
\Lambda_n^\sigma(x_1,\dots,x_n);
$
\item
$
\det\Lambda_n^\sigma(x_1,\dots,x_n)
=
\prod_{i=1}^n (\det x_i)^{1/n}.
$
\end{enumerate}
\end{proposition}

\begin{proof}
All statements follow from the corresponding properties of the Karcher mean on the two factors.

For (2), note that in any Hadamard space the Karcher mean of two equally weighted points is their midpoint.

For (4), the non-expansiveness estimate follows from the corresponding estimate for Karcher means
in Hadamard spaces, applied to $\Omega_\sigma^\pm$ and transported by the isometry $f_\sigma$.

For (5), use
$$
q_\sigma^+(\alpha_i x_i)=\alpha_i q_\sigma^+(x_i),
\qquad
q_\sigma^-(\alpha_i x_i)=q_\sigma^-(x_i),
$$
together with the homogeneity of the Karcher mean on the symmetric cone $\Omega_\sigma^+$ and the identity $f_\sigma(\lambda a,b)=\lambda f_\sigma(a,b)$.

For (6), write $x_i=f_\sigma(a_i,b_i)$. Then
$$
x_i^{-1}=f_\sigma(a_i^{-1},b_i^{-1}).
$$
Since inversion is an isometry on each Hadamard factor, the Karcher mean is equivariant under inversion, and therefore
$$
\Lambda_n^\sigma(x_1^{-1},\dots,x_n^{-1})
=
f_\sigma\bigl(\Lambda_n^+(a_1,\dots,a_n)^{-1},\Lambda_n^-(b_1,\dots,b_n)^{-1}\bigr)
=
\Lambda_n^\sigma(x_1,\dots,x_n)^{-1}.
$$
Finally, (7) follows from Lemma~\ref{lem:detone} and the determinant identity for the Karcher mean on the symmetric cone $\Omega_\sigma^+$, namely,
$$
\det\Lambda_n^\sigma(x_1,\dots,x_n)
=
\det\Lambda_n^+(a_1,\dots,a_n)
=
\prod_{i=1}^n (\det a_i)^{1/n}
=
\prod_{i=1}^n (\det x_i)^{1/n}.
$$

\end{proof}

\section{Examples: explicit Mostow decompositions}\label{sec:examples}
In this section, we illustrate the factorization and Mostow decomposition in three concrete settings: positive definite Hermitian matrices, positive definite real symmetric matrices, and the Lorentz cone (spin factors).

\subsection{The Hermitian cone}

Let $V=\Herm(r,\bC)$, with Jordan product $X\circ Y=\frac12(XY+YX)$ and inner product $\langle X,Y\rangle=\operatorname{tr}(XY)$. In this case, $P(X)Y=XYX$, and the corresponding symmetric cone is
$$
\Omega=\Herm^{++}(r,\bC):= \{ X \in \operatorname{Herm}(r,\mathbb{C}) : X \succ 0 \},
$$
the cone of positive definite Hermitian matrices.

If $Q$ is a Hermitian unitary matrix satisfying $Q^2=I$, then
$$
\sigma(X)=P(Q)X=QXQ
$$
is a special involution. One has
$$
V^+=\{X\in \Herm(r,\bC): QXQ=X\},
\qquad
V^-=\{X\in \Herm(r,\bC): QXQ=-X\},
$$
and
$$
\Omega_\sigma^+=\{A\in \Omega : QAQ=A\},
\qquad
\Omega_\sigma^-=\{B\in \Omega : QBQ=B^{-1}\}.
$$
The factorization theorem becomes
$$
X=A^{1/2}BA^{1/2},
$$
with inverse
$$
X \longmapsto \bigl(X\gm (QXQ),\ X\sgm (QX^{-1}Q)\bigr).
$$
This recovers the matrix factorization from \cite{Lim}.

\subsection{The real symmetric cone}

Let $V=\Sym(r,\bR)$, with Jordan product $x\circ y=\frac12(xy+yx)$ and inner product $\inner{x}{y}=\operatorname{tr}(xy)$. In this case, $P(x)y=xyx$, and the corresponding symmetric cone is
$$\Omega=\Sym^{++}(r,\bR):= \{ x \in \operatorname{Sym}(r,\mathbb{R}) : x \succ 0 \},
$$
the cone of positive definite symmetric matrices.

Fix an idempotent $c\in V$. Since $c=c^\top$ and $c^2=c$, it is the orthogonal projection onto a subspace $E\subset\mathbb R^r$. Choose an orthonormal basis in which
 $$
c=\begin{pmatrix} I_p & 0\\ 0 & 0\end{pmatrix},
\qquad p=\mathrm{rank}(c),\qquad q=r-p.
$$
Then
$$
w=2c-e
=
\begin{pmatrix} I_p & 0\\ 0 & -I_q\end{pmatrix}
=:J_{p,q},
\qquad w^2=I_r.
$$
Hence the special involutive automorphism $\sigma=P(w)$ acts by congruence
$$
\sigma(x)=P(w)x=J_{p,q}xJ_{p,q},\qquad x\in\Sym(r,\mathbb R).
$$
In fact, $\sigma$ is a \emph{Cartan involution}, meaning that the bilinear form $(x,y)\mapsto\inner{x}{\sigma(y)}$ is positive definite; see \cite{Kayoya,Helwig}.

Write any $x\in V$ in block form relative to $\mathbb R^r=E\oplus E^\perp$:
$$
x=
\begin{pmatrix}
A & B\\
B^\top & D
\end{pmatrix},
\qquad A\in\Sym(p,\mathbb R),\ D\in\Sym(q,\mathbb R),\ B\in\mathbb R^{p\times q}.
$$
Then
$$
\sigma(x)=J_{p,q}xJ_{p,q}=
\begin{pmatrix}
A & -B\\
-B^\top & D
\end{pmatrix}.
$$
Therefore
\begin{eqnarray*}
V^+=\{x:\sigma(x)=x\}
&=&
\left\{
\begin{pmatrix} A & 0\\ 0 & D\end{pmatrix}
: A\in\Sym(p,\mathbb R),\ D\in\Sym(q,\mathbb R)
\right\}\\
&\simeq& \Sym(p,\mathbb R)\oplus \Sym(q,\mathbb {R})
\end{eqnarray*}
and
\begin{eqnarray*}
V^-=\{x:\sigma(x)=-x\}
&=&
\left\{
\begin{pmatrix} 0 & B\\ B^\top & 0\end{pmatrix}
: B\in\mathbb R^{p\times q}
\right\}\\
&\simeq& \operatorname{Mat}(p\times q,\mathbb R).
\end{eqnarray*}
In particular, $V^+$ is a Euclidean Jordan algebra, though not necessarily simple, and $V^-$ is a Jordan triple system.

Since $V^+$ consists of block diagonal symmetric matrices, we obtain
\begin{eqnarray*}
\Omega_\sigma^+=\exp V^+
&=&
\left\{
\begin{pmatrix}
A & 0\\
0 & D
\end{pmatrix}
: A\in\Sym^{++}(p,\mathbb R),\ D\in\Sym^{++}(q,\mathbb R)
\right\}\\
&\simeq& \Sym^{++}(p,\mathbb R)\times \Sym^{++}(q,\mathbb R),
\end{eqnarray*}
which is a product of symmetric cones.

On the other hand, an element $b\in\Omega$ belongs to $\Omega_\sigma^-$ if and only if $\sigma(b)=b^{-1}$, that is, $wbw=b^{-1}$. Since $w^2=I$, this is also equivalent to $bwb=w$. Thus
$$
\Omega_\sigma^-
=
\{b\in\Sym^{++}(r,\mathbb R): b w b=w\}
=
\Sym^{++}(r,\mathbb R)\cap O(p,q),
$$
the set of positive definite symmetric matrices that preserve the quadratic form of signature $(p,q)$. It is therefore a Makarevi\v c symmetric space of noncompact type; see \cite{Bertram}.\\

The following   standard functional  identity will be used in the sequel.
\begin{lemma}\label{lem:func}
Let $S\in \bR^{p\times q}$ and let $f$ be continuous on a neighborhood of $\operatorname{spec}(SS^\top)\cup \operatorname{spec}(S^\top S)$. Then
$$
f(SS^\top)\,S = S\,f(S^\top S).
$$
\end{lemma}

\begin{proof}
Take a singular value decomposition $S=U\Sigma V^\top$. Then
$$
SS^\top = U(\Sigma\Sigma^\top)U^\top,
\qquad
S^\top S = V(\Sigma^\top \Sigma)V^\top.
$$
Thus
$$
f(SS^\top)\,S = U f(\Sigma\Sigma^\top)\Sigma V^\top,
\qquad
S\,f(S^\top S)=U\Sigma f(\Sigma^\top \Sigma)V^\top.
$$
Because $\Sigma$ is diagonal in rectangular form, one has $f(\Sigma\Sigma^\top)\Sigma=\Sigma f(\Sigma^\top\Sigma)$, so the two expressions agree.
\end{proof}

\begin{proposition}\label{prop:realmatrix}
Let
$$
x=
\begin{pmatrix}
A & B\\
B^\top & D
\end{pmatrix}
\in \Sym^{++}(p+q,\bR).
$$
Set
$$
S:=A^{-1/2}BD^{-1/2},
\qquad
C:=(I_p-SS^\top)^{-1/2},
\qquad
E:=(I_q-S^\top S)^{-1/2},
$$
$$
\widetilde F:=CS=SE,
\qquad
\widetilde b:=
\begin{pmatrix}
C & \widetilde F\\
\widetilde F^\top & E
\end{pmatrix},
\qquad
m:=
\begin{pmatrix}
A^{1/2}C^{-1/2} & 0\\
0 & D^{1/2}E^{-1/2}
\end{pmatrix}.
$$
Then $\widetilde b\in \Omega_\sigma^-$ and
$$
x=m\widetilde b\,m^\top.
$$
If $x=a^{1/2}ba^{1/2}$ is the canonical factorization from Corollary~\ref{cor:mostow}, then
$$
a=mm^\top=
\begin{pmatrix}
A\gm (A-BD^{-1}B^\top) & 0\\
0 & D\gm (D-B^\top A^{-1}B)
\end{pmatrix},
$$
$$
b=a^{-1/2}xa^{-1/2}.
$$
Moreover, if $m=a^{1/2}k$ is the polar decomposition of $m$, then $k\in \OO(p)\times \OO(q)$ and
$$
b=k\widetilde b\,k^\top.
$$
\end{proposition}

\begin{proof}
Since $x\succ 0$, the Schur complements show that $A\succ 0$, $D\succ 0$, and
$$\begin{aligned}
I_q-S^\top S &=D^{-1/2}(D-B^{\top} A^{-1} B)D^{-1/2}\succ 0,\\
I_p-SS^\top &=A^{-1/2}(A-B D^{-1} B^{\top})A^{-1/2}\succ 0.
\end{aligned}
$$
Hence $C$ and $E$ are well defined. The relation $\widetilde F=CS=SE$ follows from Lemma~\ref{lem:func} with $f(t)=(1-t)^{-1/2}$.

We next show that $\widetilde b\succ0$. Since $C\succ0$, the Schur complement criterion reduces the claim to proving
$
E-\widetilde F^\top C^{-1}\widetilde F \succ 0.
$
Using $\widetilde F=CS$, we get
$
\widetilde F^\top C^{-1}\widetilde F
=
S^\top C S.
$
Since $CS=SE$, we also have $S^\top C = E S^\top$, and therefore
$
S^\top C S = E S^\top S.
$
Thus
$$
E-\widetilde F^\top C^{-1}\widetilde F
=
E-ES^\top S
=
E(I_q-S^\top S)
=
E^{-1}\succ 0.
$$
Hence $\widetilde b\succ 0$.

Next,
$$
C^2-\widetilde F\widetilde F^\top
=
C(I_p-SS^\top)C
=
I_p,
$$
$$
E^2-\widetilde F^\top \widetilde F
=
E(I_q-S^\top S)E
=
I_q,
$$
and, by Lemma~\ref{lem:func} with $f(t)=(1-t)^{-1}$,
$$
C\widetilde F=\widetilde F E.
$$
Therefore
$
\widetilde b J_{p,q}\widetilde b = J_{p,q},
$
so $\widetilde b\in \Omega_\sigma^-$.

A direct block multiplication gives
$$
m\widetilde b\,m^\top
=\left(\begin{array}{cc}
A & A^{1 / 2} C^{-1 / 2} \widetilde{F} E^{-1 / 2} D^{1 / 2} \\
D^{1 / 2} E^{-1 / 2} \widetilde{F}^{\top} C^{-1 / 2} A^{1 / 2} & D
\end{array}\right)
$$
For the off-diagonal block, write
$$
A^{1/2}C^{-1/2}\widetilde F E^{-1/2}D^{1/2}
=
A^{1/2}C^{1/2} S E^{-1/2}D^{1/2},
$$
and Lemma~\ref{lem:func}, applied with $f(t)=(1-t)^{-1/4}$, gives
$
C^{1/2}S = S E^{1/2}.
$
Therefore
$$
A^{1/2}C^{1/2} S E^{-1/2}D^{1/2}
=
A^{1/2}S D^{1/2}
=
B.
$$
Hence
$$
m\widetilde b\,m^\top
=
\begin{pmatrix}
A&B\\
B^\top&D
\end{pmatrix}
=x.
$$

Now $\sigma(m)=m$ because $m$ is block diagonal, and $\sigma(\widetilde b)=\widetilde b^{-1}$ because $\widetilde b\in \Omega_\sigma^-$. Hence
$$
\sigma(x)=m\widetilde b^{-1}m^\top.
$$
By congruence invariance of the matrix geometric mean,
$$
x\gm \sigma(x)=m(\widetilde b \gm \widetilde b^{-1})m^\top=mm^\top.
$$
Since $\widetilde b \gm \widetilde b^{-1}=I$, this gives the formula for $a=mm^\top$. The displayed expression for $a$ follows from
$$
A^{1/2}C^{-1}A^{1/2}=A^{1/2}(I_p-SS^\top)^{1/2}A^{1/2}
=
A\gm (A-BD^{-1}B^\top),
$$
and similarly for the lower-right block.

Finally,
$$
x=m\widetilde b\,m^\top=a^{1/2}k\widetilde b\,k^\top a^{1/2},
$$
and $k$ commutes with $J_{p,q}$ because $k$ is block diagonal. Thus $k\widetilde b\,k^\top\in \Omega_\sigma^-$, and uniqueness of the factorization implies
$$
b=k\widetilde b\,k^\top.
$$
The orthogonal matrix $k$ can also be computed explicitly. One obtains

$$
\begin{aligned}
k & =\left(\begin{array}{cc}
\left(A^{1 / 2} C^{-1} A^{1 / 2}\right)^{-1 / 2} A^{1 / 2} C^{-1 / 2} & 0 \\
0 & \left(D^{1 / 2} E^{-1} D^{1 / 2}\right)^{-1 / 2} D^{1 / 2} E^{-1 / 2}
\end{array}\right) \\
& =\left(\begin{array}{cc}
A^{-1 / 2}\left(A \gm \left(A-B D^{-1} B^{\top}\right) C^{-1 / 2}\right. & 0 \\
0 & D^{-1 / 2}\left(D \gm \left(D-B^{\top} A^{-1} B\right) E^{-1 / 2}\right.
\end{array}\right)
\end{aligned}
$$
\end{proof}

\subsection{The Lorentz cone}

Let $m\ge 1$ and $V=\bR\oplus \bR^m$ with Jordan product
$$
(\lambda,u)\circ (\mu,v)=\bigl(\lambda\mu+\inner{u}{v},\ \lambda v+\mu u\bigr), \quad e=(1,0).
$$
Then $V$ is the rank two Euclidean Jordan algebra usually called a \emph{spin factor}, equipped with the inner product\footnote{We fix here the normalization
$
(x\mid y)=\frac12\tr_V(x\circ y),
$
which is one half of the usual trace inner product. }
\begin{equation}\label{inner-spin-factor}
((\lambda, u) \mid(\mu, v))=\lambda \mu+\inner{u}{v}.
\end{equation}
The corresponding symmetric cone is the Lorentz cone
$$
\Omega=\{(\lambda,u)\in\bR\oplus \bR^m \;:\; \lambda>\|u\|\}.
$$
 The trace, determinant, and inverse are
$$
\tr_V(\lambda,u)=2\lambda, \quad \det_V(\lambda,u)=\lambda^2-\|u\|^2,
\qquad
(\lambda,u)^{-1}
=
\frac{1}{\lambda^2-\|u\|^2}(\lambda,-u).
$$
A direct computation from \(P(x)=2L(x)^2-L(x^2)\) gives
\begin{equation}\label{quad-spin}
P(\lambda,u)(\mu,v)
=
\Bigl((\lambda^2+\|u\|^2)\mu+2\lambda\langle u,v\rangle,\
(\lambda^2-\|u\|^2)v+2\lambda\mu\,u+2\langle u,v\rangle u\Bigr).
\end{equation}

We now derive explicit formulas for the metric and spectral geometric means in the spin factor.

For $x=(\lambda,u)$ and $y=(\mu,v)$ in $\Omega$, define
$$\Delta_x:=\det(x)=\lambda^2-\|u\|^2,\qquad
\Delta_y:=\det(y)=\mu^2-\|v\|^2,$$
$$\check{x}:=(\lambda,-u),\qquad \check{y}:=(\mu,-v),$$
and
$$\beta(x,y):=(x\mid \check{y})=\lambda\mu-\langle u,v\rangle.$$
 The linearization of the determinant gives
$$\det(ax+by)=a^2\Delta_x+b^2\Delta_y+2ab\beta(x,y).$$

\begin{proposition}\label{prop:spingm}
 
The metric geometric mean in the Lorentz cone is given by
\begin{equation}\label{gm-spin-factor}
x\gm y
=
\frac{\sqrt{\Delta_y}\,x+\sqrt{\Delta_x}\,y}
{\sqrt{2\bigl(\beta(x,y)+\sqrt{\Delta_x\Delta_y}\bigr)}}. 
\end{equation}
That is,
$$
x\gm y
=
\left(
\frac{\sqrt{\Delta_y}\,\lambda+\sqrt{\Delta_x}\,\mu}
{\sqrt{2\bigl(\lambda\mu-\inner{u}{v}+\sqrt{\Delta_x\Delta_y}\bigr)}},
\frac{\sqrt{\Delta_y}\,u+\sqrt{\Delta_x}\,v}
{\sqrt{2\bigl(\lambda\mu-\inner{u}{v}+\sqrt{\Delta_x\Delta_y}\bigr)}}
\right).
$$
\end{proposition}

\begin{proof}
One checks directly from \eqref{quad-spin} that
$$
P(z)w=2(z\mid w)z-\det(z)\check{w}
$$
for all $z,w\in V$. Set
$$
N:=\sqrt{\Delta_y}\,x+\sqrt{\Delta_x}\,y,
\qquad
\delta:=2\bigl(\beta(x,y)+\sqrt{\Delta_x\Delta_y}\bigr),
\qquad
z:=\frac{N}{\sqrt{\delta}}.
$$
A short calculation gives
$$
\det(z)=\sqrt{\Delta_x\Delta_y},
\qquad
(z\mid x^{-1})=\frac{\sqrt{\delta}}{2\sqrt{\Delta_x}}.
$$
Since $x^{-1}=\check{x}/\Delta_x$, it follows that
$$
P(z)x^{-1}=2(z\mid x^{-1})z-\det(z)\check{x^{-1}}=y.
$$
By Proposition~\ref{prop:metricmean}(1), $z=x\gm y$.
\end{proof}

When $\det(x)=\det(y)=1$, one has
$$
x\gm y=\frac{x+y}{\sqrt{\det(x+y)}}.
$$
This yields the following alternative expression for the metric geometric mean,
\begin{align*}
x\gm y
&=(\Delta_x\Delta_y)^{1/4}
\frac{\Delta_x^{-1/2}x+\Delta_y^{-1/2}y}
{\sqrt{\det\bigl(\Delta_x^{-1/2}x+\Delta_y^{-1/2}y\bigr)}}\\
&=(\Delta_x\Delta_y)^{1/4}
\frac{\sqrt{\Delta_y}\,x+\sqrt{\Delta_x}\,y}
{\sqrt{\det\bigl(\sqrt{\Delta_y}\,x+\sqrt{\Delta_x}\,y\bigr)}}.
\end{align*}

For $x=(\lambda,u)$ and $y=(\mu,v)$ in $\Omega$, set
$$
a:=\Delta_x,
\qquad
b:=\Delta_y,
\qquad
\gamma:=\sqrt{\frac{\lambda\mu+\inner{u}{v}+\sqrt{\Delta_x\Delta_y}}{2}}.
$$

\begin{proposition}\label{prop:spinsgm}
The spectral geometric mean in the Lorentz cone  is given by
$$
x\sgm y
=
\left(
\gamma,
\frac{(ab)^{1/4}(\mu u+\lambda v)+\gamma(\sqrt{b}\,u+\sqrt{a}\,v)}
{\lambda\sqrt{b}+\mu\sqrt{a}+2\gamma (ab)^{1/4}}
\right).
$$
\end{proposition}

\begin{proof}
Set $q:=x^{-1}\gm y$. By Proposition~\ref{prop:spingm},
$$
q
=
\frac{1}{2\sqrt{a}\,\gamma}
\Bigl(\sqrt{b}\,\lambda+\sqrt{a}\,\mu,\ -\sqrt{b}\,u+\sqrt{a}\,v\Bigr).
$$
The square root formula in the Lorentz cone is
$$
(\alpha,w)^{1/2}
=
\left(
\sqrt{\frac{\alpha+\sqrt{\alpha^2-\|w\|^2}}{2}},
\frac{w}{\sqrt{2(\alpha+\sqrt{\alpha^2-\|w\|^2})}}
\right).
$$
Since $\det(q)=\sqrt{b/a}$, this gives
$$
q^{1/2}
=
\left(
\frac{\sqrt{d}}{2a^{1/4}\sqrt{\gamma}},
\frac{-\sqrt{b}\,u+\sqrt{a}\,v}{2a^{1/4}\sqrt{\gamma}\sqrt{d}}
\right),
$$
where
$$
d:=\lambda\sqrt{b}+\mu\sqrt{a}+2\gamma (ab)^{1/4}.
$$
Substituting this into the definition
$$
x\sgm y=P(q^{1/2})x
$$
and simplifying with \eqref{quad-spin} yields the asserted formula.
\end{proof}

Fix a unit vector $u_0\in \bR^m$, and set
$$
c:=\frac12(1,u_0),
\qquad
w:=2c-e=(0,u_0).
$$
Then $c$ is a primitive idempotent, $w^2=e$, and $\sigma:=P(w)$ is a special involution. From \eqref{quad-spin},
$$
\sigma(\mu,v)=P(0,u_0)(\mu,v)=\bigl(\mu,\ -v+2\inner{u_0}{v}u_0\bigr).
$$
Thus $\sigma=(I_{\mathbb R},g)$, where $g$ acts on $\mathbb R^m$ as the
orthogonal reflection across the line $\mathbb R u_0$.

Decompose $v\in\mathbb R^m=\mathbb R u_0\oplus u_0^\perp$ as $v=v_\parallel+v_\perp$, where
$$
v_\parallel=\langle u_0,v\rangle u_0,
\qquad
v_\perp\perp u_0.
$$
Then
$$
\sigma(\mu,v)=(\mu,v_\parallel-v_\perp).
$$
Hence
$$
V^+=\{(\lambda,t\,u_0): \lambda,t\in \bR\},
\qquad
V^-=\{(0,v): v\perp u_0\}.
$$

The Jordan exponential is
$$
\exp(\lambda,u)=\bigl(e^\lambda \cosh r,\ e^\lambda \tfrac{\sinh r}{r}u\bigr),
\qquad
r:=\|u\|,
$$
with the convention $\frac{\sinh r}{r}=1$ at $r=0$. Therefore
$$
\Omega_\sigma^+
=
\{(\alpha,\beta u_0): \alpha>|\beta|\},
$$
which is a two-dimensional Lorentz cone {\color{blue}in} $\mathbb R e\oplus\mathbb R u_0$.  Under the  isomorphism
$$
(\alpha,\beta u_0)\longmapsto(\alpha+\beta,\alpha-\beta),
$$
it is identified with the symmetric space
$$
(\mathbb R_{>0})^2
\simeq
\left(\GL^+(1,\mathbb R)/\SO(1)\right)^2.
$$
On the other hand,
$$
\Omega_\sigma^-
=
\{(\alpha,v): \alpha>0,\ v\perp u_0,\ \alpha^2-\|v\|^2=1\}.
$$
Thus $\Omega_\sigma^-$ is the upper sheet of a two-sheeted hyperboloid; in particular, it is the real hyperbolic space $\SO_0(1,m-1)/\SO(m-1)$.

Now let $x=(\lambda,u)\in\Omega$ and decompose
$$
u=\xi u_0+v,
\qquad
\xi:=\inner{u}{u_0},
\qquad
v\perp u_0.
$$
Set
$$
\Delta:=\det x=\lambda^2-\xi^2-\|v\|^2,
\qquad
S:=\lambda^2-\xi^2.
$$
Define
$$
a:=\left(\frac{\sqrt{\Delta}\,\lambda}{\sqrt{S}},\ \frac{\sqrt{\Delta}\,\xi}{\sqrt{S}}\,u_0\right),
\qquad
b:=\left(\frac{\sqrt{S}}{\sqrt{\Delta}},\ \frac{v}{\sqrt{\Delta}}\right).
$$

\begin{proposition}\label{prop:spinfactorization}
One has $a\in \Omega_\sigma^+$, $b\in \Omega_\sigma^-$, and
$$
x=P(a^{1/2})\,b.
$$
Consequently,
$$
x\gm \sigma(x)=\left(\frac{\sqrt{\Delta}\,\lambda}{\sqrt{S}},\ \frac{\sqrt{\Delta}\,\xi}{\sqrt{S}}\,u_0\right),
\qquad
x\sgm \sigma(x^{-1})=\left(\frac{\sqrt{S}}{\sqrt{\Delta}},\ \frac{v}{\sqrt{\Delta}}\right).
$$
\end{proposition}

\begin{proof}
Since $\lambda>\|u\|=\sqrt{\xi^2+\|v\|^2}\ge |\xi|$, the point $a$ belongs to $\Omega_\sigma^+$. Also, the vector part of $b$ is orthogonal to $u_0$ and
$$
\left(\frac{\sqrt{S}}{\sqrt{\Delta}}\right)^2-\left\|\frac{v}{\sqrt{\Delta}}\right\|^2
=
\frac{S-\|v\|^2}{\Delta}
=
1,
$$
so $b\in \Omega_\sigma^-$.

Write
$$
a^{1/2}=(r,su_0).
$$
Then
$
(r,su_0)^2=(r^2+s^2,\ 2rs\,u_0),
$
hence
$$
r^2+s^2=\frac{\sqrt{\Delta}\,\lambda}{\sqrt{S}},
\qquad
2rs=\frac{\sqrt{\Delta}\,\xi}{\sqrt{S}},
\qquad
r^2-s^2=\sqrt{\det(a)}=\sqrt{\Delta}.
$$
Using \eqref{quad-spin} and $v\perp u_0$, we get
$$
P(r,su_0)(\mu,v)=\bigl((r^2+s^2)\mu,\ (r^2-s^2)v+2rs\,\mu\,u_0\bigr).
$$
Applying this to $b$ gives
$$
P(a^{1/2})b
=
\left(
\frac{\sqrt{\Delta}\,\lambda}{\sqrt{S}}\cdot \frac{\sqrt{S}}{\sqrt{\Delta}},
\ \sqrt{\Delta}\cdot \frac{v}{\sqrt{\Delta}}+\frac{\sqrt{\Delta}\,\xi}{\sqrt{S}}\cdot \frac{\sqrt{S}}{\sqrt{\Delta}}\,u_0
\right)
=
(\lambda,\xi u_0+v)
=
x.
$$
The final formulas follow from Theorem~\ref{thm:factorization}.
\end{proof}

\section*{Acknowledgments}
K. Koufany gratefully acknowledges support from ``Lorraine Université d'Excellence,'' part of the France 2030 program (ANR-15-IDEX-04-LUE). The work of Y. Lim was supported by a National Research Foundation of Korea (NRF) grant funded by the Korean government (MSIT), grant no.~RS-2025-14842993.  The authors also acknowledge the support of the South Korea-France research cooperation PHC STAR program.

\end{document}